\documentclass[final]{siamltex}
\usepackage{graphicx}
\usepackage{url}
\usepackage{mathptmx}     
\usepackage[english]{babel}
\usepackage{amssymb,amsmath}
\usepackage{enumitem}
\usepackage{color}
\usepackage{hyperref}

\newtheorem{remark}[theorem]{Remark}

\newcommand\sinc{\sin\!{\rm{c}}}
\newcommand\e{\mathrm{e}}
\newcommand{\R}{\mathbb{R}}

\newcommand{\bigo}[1]{\mathcal{O}(#1)}
\providecommand{\norm}[1]{\ensuremath{\lVert#1\rVert}}
\providecommand{\Norm}[1]{\ensuremath{\Big\lVert#1\Big\rVert}}

\title{A trigonometric method for the linear stochastic wave equation}

\author{David Cohen\thanks{Institut f\"ur Angewandte und Numerische Mathematik, 
              Karlsruher Institut f\"ur Technologie,
              DE-76128 Karlsruhe, Germany.
              {\tt david.cohen@kit.edu}.
              Present address: Matematik och matematisk statistik, 
              Ume{\aa} universitet, SE-90187 Ume{\aa}, 
              Sweden. {\tt david.cohen@math.umu.se}} 
       \and Stig Larsson\thanks{Department of Mathematical Sciences, 
              Chalmers University of Technology and University 
              of Gothenburg, SE-412 96 Gothenburg, Sweden.
              {\tt stig@chalmers.se}} 
       \and Magdalena Sigg\thanks{Mathematisches Institut, Universit\"at Basel, 
              CH-4051 Basel, Switzerland.
              {\tt Magdalena.Sigg@unibas.ch}}
         }

\begin{document}

\maketitle

\begin{abstract}
A fully discrete approximation of the linear 
stochastic wave equation driven by additive noise is 
presented. A standard finite element method 
is used for the spatial discretisation 
and a stochastic trigonometric scheme for the 
temporal approximation. 
This explicit time integrator allows 
for error bounds independent of the space 
discretisation and thus do not have a step size restriction as 
in the often used St\"ormer-Verlet-leap-frog scheme. 
Moreover it enjoys a trace formula as does the exact solution 
of our problem. These favourable properties are 
demonstrated with numerical experiments.
\end{abstract}

\begin{keywords}
Stochastic wave equation, Additive noise,  
Strong convergence, Trace formula, Stochastic 
trigonometric schemes, Geometric numerical integration
\end{keywords}

\begin{AMS}
65C20, 60H10, 60H15, 60H35, 65C30
\end{AMS}

\pagestyle{myheadings}
\thispagestyle{plain}
\markboth{D. Cohen and S. Larsson and M. Sigg}{A trigonometric method for the stochastic wave equation}


\section{Introduction}\label{sect:intro}

We consider the numerical discretisation of the 
linear stochastic wave equation with additive noise
\begin{equation}
\begin{aligned}
&\mathrm{d}\dot{u} - \Delta u\, \mathrm{d}t = \mathrm{d}W &&\quad \mathrm{in} \ 
\mathcal{D}\times(0,\infty), \\
&u = 0 &&\quad \mathrm{in} \ \partial\mathcal{D}\times(0,\infty), \\
&u(\cdot,0) = u_0, \ \dot{u}(\cdot,0)=v_0 &&\quad \mathrm{in} \ \mathcal{D},
 \end{aligned}
\label{swe}
\end{equation}
where $u=u(x,t)$, $\mathcal{D}\subset \mathbb{R}^d$, $d=1,2,3$, is a
bounded convex domain with polygonal boundary $\partial\mathcal D$,
and the dot ``$\cdot$'' stands for the time derivative.  The
stochastic process $\{W(t)\}_{t\geq 0}$ is an
$L_2(\mathcal{D})$-valued $Q$-Wiener process with respect to a
normal filtration $\{\mathcal{F}_t\}_{t\geq 0}$ on a filtered
probability space $(\Omega,\mathcal{F},\mathbb{P},\{\mathcal{F}_t\}_{t\geq 0})$. 
The initial data $u_0$ and $v_0$ are $\mathcal{F}_0$-measurable random
variables. We will numerically solve this problem with a finite
element method in space \cite{kls}
and a stochastic trigonometric method in time \cite{c10} 
and \cite{cs11} (see Section~\ref{sect:trigo}).

There are many reasons to study stochastic wave equations.  Let us
mention the motion of a suspended cable under wind loading
\cite{dsm03}; the motion of a strand of DNA in a liquid
\cite{dkmnx09}; or the motion of shock waves on the surface of the sun
\cite{dkmnx09}.  All these stochastic partial differential equations
are of course nonlinear and highly nontrivial. But in order to derive
efficient numerical schemes, we first look at model problems
like \eqref{swe}.  

The numerical analysis of the stochastic wave equation is only in its
beginning in comparison with the numerical analysis of parabolic
problems. We refer to \cite{cao07} and \cite{sch08} for spectral-type 
(spatial) discretisations of our stochastic partial differential 
equation and to the introduction of \cite{kls} 
for other types of spatial discretisations. 
We now comment on works dealing with the time
discretisation of \eqref{swe}. Strong convergence estimates for 
implicit one-step methods can be found in \cite{KLLweaktime}, despite 
the main theme of the paper which is weak convergence. 
Both for spatial and temporal approximation the order of convergence is found 
to be somewhat lower than the order of regularity, see Remark~\ref{order-remark} below.   
In \cite{w06} the leap-frog scheme is
applied to the nonlinear stochastic wave equation with space-time
white noise on the whole line. A strong convergence rate
$\bigo{h^{1/2}}$ is proved, where $h$ is the step size in both time
and space, which is in agreement with the order of regularity in this case. 
The reason for this is that the Green's functions of
the continuous and the discrete problems coincide at mesh points. A
similar trick is also used in \cite{mpw02} and \cite{mpw03} to derive
an ``exact'' solver. 
Let us finally mention the work \cite{klns11}, 
where error bounds in the $p$-th mean 
for general semilinear stochastic 
evolution equations are presented. 
The authors consider a Fourier Galerkin discretisation 
in space and the exponential Euler scheme in time. 
This exponential time integrator 
(see also \cite{ho10}, \cite{jk09}, \cite{lr04} and 
references therein) is, 
in the linear case, precisely 
the one that we use \cite{cs11}. 

The paper is organised as follows. 
Some preliminaries and the main results 
from \cite{kls} on strong 
convergence estimates for the finite 
element approximation of our problem 
are presented in Section~\ref{sect:space}. 
The stochastic trigonometric scheme is introduced 
in Section~\ref{sect:trigo} and a convergence analysis 
is carried out in Section~\ref{sect:strong}.  
A trace formula for the numerical 
integrator is obtained in Section~\ref{sect:trace} and 
finally in Section~\ref{sect:numexp} 
numerical experiments demonstrate 
the efficiency of our discretisation.

\section{A finite element approximation of the stochastic wave equation}\label{sect:space}

Before we can state the main result on the finite element approximation 
of \cite{kls}, 
we must define the spaces, norms and notations we will need. 
Let $U$ and $H$ be separable Hilbert spaces with 
norms $\norm{\cdot}_U$, resp. $\norm{\cdot}_H$. 
$\mathcal{L}(U,H)$ denotes the space of bounded linear operators 
from $U$ to $H$ and $\mathcal{L}_2(U,H)$ the space of Hilbert-Schmidt operators with norm
$$ \norm{T}_{\mathcal{L}_2(U,H)}:=\Big(\sum_{k=1}^\infty \norm{Te_k}_H^2\Big)^{1/2}, $$
where $\{e_k\}_{k=1}^\infty$ is an orthonormal basis of $U$. If $H=U$,
then $\mathcal{L}(U)=\mathcal{L}(U,U)$ and
$\mathrm{HS}=\mathcal{L}_2(U,U)$.  Furthermore, if
  $(\Omega,\mathcal{F},\mathbb{P},\{\mathcal{F}_t\}_{t\geq 0})$ is a
  filtered probability space, then $L_2(\Omega,H)$ is the space of
$H$-valued square integrable random variables with norm
$$ 
\norm{v}_{L_2(\Omega,H)}=\mathbb{E}[\norm{v}_H^2]^{1/2}. 
$$
Let $Q\in\mathcal L(U)$ be a self-adjoint, positive  
semidefinite operator. 
The driving stochastic process $W(t)$ in \eqref{swe} is a $U$-valued $Q$-Wiener process 
with respect to the filtration $\{\mathcal F_t\}_{t\geq0}$ and has the orthogonal 
expansion \cite[Section 2.1]{pr07}
\begin{equation}\label{fnoise}
W(t)=\sum_{j=1}^\infty\gamma_j^{1/2}\beta_j(t)e_j,
\end{equation}
where $\{(\gamma_j,e_j)\}_{j=1}^\infty$ are eigenpairs of $Q$ with orthonormal eigenvectors 
and $\{\beta_j(t)\}_{j=1}^\infty$ are real-valued mutually independent
standard Brownian motions. It is then possible to define the
stochastic integral  $\int_0^t\Phi(s)\,\mathrm{d}W(s)$ together with
It\^o's isometry, \cite{pr07}:
\begin{equation}
\mathbb{E}\Big[\Big\lVert\int_0^t\Phi(s)\,\mathrm{d}W(s)\Big\rVert_H^2\Big] 
= \int_0^t\norm{\Phi(s)Q^{1/2}}_{\mathcal{L}_2(U,H)}^2\,\mathrm{d}s,
\label{ito}
\end{equation}
where $\Phi:[0,\infty)\rightarrow \mathcal{L}(U,H)$ is such that the right side is finite. 

For the stochastic wave equation \eqref{swe}, we define $U=L_2(\mathcal{D})$ 
and $\Lambda = -\Delta$ with $D(\Lambda) = H^2(\mathcal{D})\cap H_0^1(\mathcal{D})$. 
We assume that the covariance operator $Q$ of $W$ satisfies  
\begin{equation} \label{Qassumption}
\norm{\Lambda^{(\beta-1)/2}Q^{1/2}}_\mathrm{HS}<\infty
\end{equation}
for some $\beta\ge0$ and with the Hilbert-Schmidt norm defined above. 
If $Q$ is of trace class, i.\,e., $\mathrm{Tr}(Q)=\norm{Q^{1/2}}_{\mathrm{HS}}^2<\infty$, 
then $\beta=1$. If $Q=\Lambda^{-s}$, $s\geq0$, then $\beta<1+s-d/2$. This 
follows from the asymptotic behaviour of the eigenvalues 
of $\Lambda$, $\lambda_j\sim j^{2/d}$. 
In particular, if $Q=I$, then $\beta<\frac{1}{2}$ and $d=1$. Note that
we do not assume that $\Lambda$ and $Q$ have a common eigenbasis.

We will use the spaces $\dot{H}^\alpha = D(\Lambda^{\alpha/2})$ for $\alpha\in\mathbb{R}$. 
The corresponding norm is given by
$$ 
\norm{v}_\alpha := \norm{\Lambda^{\alpha/2}v}_{L_2(\mathcal{D})}
=
\Big(\sum_{j=1}^\infty\lambda_j^\alpha(v,\varphi_j)_{L_2(\mathcal{D})}^2\Big)^{1/2}, 
$$
where $\{(\lambda_j,\varphi_j)\}_{j=1}^\infty$ are the eigenpairs of $\Lambda$ 
with orthonormal eigenvectors. We also write 
$H^\alpha=\dot H^\alpha\times \dot H^{\alpha-1}$ 
and $H=H^0=\dot H^0\times \dot H^{-1}$.

We use a standard piecewise linear finite element method for the 
spatial discretisation. Let $\{\mathcal{T}_h\}$ be a quasi-uniform 
family of triangulations of $\mathcal{D}$ with $h_K = \operatorname{diam}(K)$, 
$h=\max_{K\in\mathcal{T}_h}h_K$, and denote by $V_h$ the space of 
piecewise linear continuous functions with respect to 
$\mathcal{T}_h$ which vanish on $\partial\mathcal{D}$. 
Hence, $V_h\subset H^1_0(\mathcal{D})=\dot{H}^1$.

We introduce discrete variants of $\norm{\cdot}_\alpha$ and $\dot{H}^\alpha$:
$$ \norm{v_h}_{h,\alpha} =
\norm{\Lambda_h^{\alpha/2}v_h}_{L_2(\mathcal{D})}, \ \ v_h\in V_h;
\quad \dot{H}_h^\alpha=V_h \ \text{equipped with}\ \norm{\cdot}_{h,\alpha},
$$
where $\Lambda_h:V_h\rightarrow V_h$ is the discrete Laplace operator
defined by 
\begin{align*}
(\Lambda_hv_h,w_h)_{L_2(\mathcal{D})}=(\nabla v_h,\nabla w_h)_{L_2(\mathcal{D})},\quad\forall w_h \in V_h.
\end{align*}

Denoting the velocity of the solution by $u_2:=\dot u_1:=\dot u$, 
one can rewrite \eqref{swe} as
\begin{equation}
\begin{split}
& \mathrm{d}X(t) = AX(t)\,\mathrm{d}t+B\,\mathrm{d}W(t), \ \ t>0, \\
& X(0) = X_0,
\label{swe2}
\end{split}
\end{equation}
where $ A := \begin{bmatrix} 0 & I \\ -\Lambda & 0 \end{bmatrix}$,
$B:=\begin{bmatrix} 0 \\ I \end{bmatrix}$, $X := \begin{bmatrix} u_1
  \\ u_2 \end{bmatrix}$ and $X_0 := \begin{bmatrix} u_0 \\
  v_0 \end{bmatrix}$. The operator $A$ with $D(A)=H^1=\dot
  H^1\times \dot H^{0}$ is the generator of a strongly continuous
semigroup of bounded linear
operators $E(t)=\e^{tA}$ on $H^0=\dot{H}^0\times \dot{H}^{-1}$, in fact, 
a unitary group. 

Let $\mathcal{P}_h:\dot{H}^0\rightarrow V_h$ and $\mathcal{R}_h:\dot{H}^1\rightarrow V_h$ 
denote the orthogonal projectors onto the finite element space $V_h\subset H_0^1(\mathcal{D}) = \dot{H}^1$, 
where we recall that $V_h$ is the space of piecewise linear continuous functions. 
The finite element approximation of \eqref{swe} can then be written as
\begin{equation}
\begin{split}
& \mathrm{d}\dot{u}_{h,1}(t) + \Lambda_h u_{h,1}(t)\,\mathrm{d}t = \mathcal{P}_h\,\mathrm{d}W(t), \ \ t>0, \\
& u_{h,1}(0)=u_{h,0}, \ u_{h,2}(0)=v_{h,0},
\label{femswe1}
\end{split}
\end{equation}
or in the abstract form 
\begin{equation}
\begin{split}
& \mathrm{d}X_h(t) = A_hX_h(t)\,\mathrm{d}t+\mathcal{P}_hB\,\mathrm{d}W(t), \ \ t>0, \\
& X_h(0) = X_{h,0},
\label{femswe2}
\end{split}
\end{equation}
where $ A_h := \begin{bmatrix} 0 & I \\ -\Lambda_h & 0 \end{bmatrix}$, 
$X_h := \begin{bmatrix} u_{h,1} \\ u_{h,2} \end{bmatrix}$ and 
$X_{h,0} := \begin{bmatrix} u_{h,0} \\ v_{h,0} \end{bmatrix}$ 
with $u_{h,0}, v_{h,0}\in V_h$. 
Again, $A_h$ is the generator of a $C_0$-semigroup $E_h(t)=\e^{tA_h}$ 
on $V_h\times V_h$. 

It is known, see, e.\,g., \cite[Example 5.8]{dz92} and
\cite{kls}, that under assumption \eqref{Qassumption} the linear
stochastic wave equation \eqref{swe2} has a unique weak solution given
by
\begin{equation}
X(t) = E(t)X_0 + \int_0^tE(t-s)B\,\mathrm{d}W(s),
\label{exactsol}
\end{equation}
with mean-square regularity of order $\beta$, 
\begin{align}   \label{weaksolutionregularity}
  \lVert X(t)\rVert_{L_2(\Omega,H^\beta)}
    \le C \Big(\lVert X_0\rVert_{L_2(\Omega,H^\beta)}+t^{1/2}
      \lVert\Lambda^{(\beta-1)/2}Q^{1/2}\rVert_{\mathrm{HS}}\Big),\quad t\ge 0.
\end{align}
Similarly, the unique solution of the finite element problem \eqref{femswe2} is given by 
\begin{equation}
X_h(t) = E_h(t)X_{h,0} + \int_0^tE_h(t-s)\mathcal{P}_hB\,\mathrm{d}W(s).
\label{exactsolfem}
\end{equation}

We quote the following theorem on the convergence of the spatial approximation.

\begin{theorem}[Theorem~5.1 in \cite{kls}]\label{thm:fem} 
Assume that $Q$ satisfies \eqref{Qassumption} for some $\beta\in[0,4]$. 
Let $X_0=[u_0,v_0]^T\in H^\beta=\dot
  H^\beta\times\dot H^{\beta-1}$, $X = [u_1,u_2]^T$ and
  $X_h=[u_{h,1},u_{h,2}]^T$ be given by \eqref{exactsol} and
  \eqref{exactsolfem}, respectively.  Then the following estimates
  hold for $t\geq 0$, where $C(t)$ is an increasing function of the
  time $t$.
\begin{itemize}[leftmargin=*]
\item If $u_{h,0} = \mathcal{P}_hu_0$, $v_{h,0} = \mathcal{P}_hv_0$ and $\beta\in[0,3]$, then
\begin{equation*}
\norm{u_{h,1}(t)-u_1(t)}_{L_2(\Omega,\dot{H}^0)} \leq C(t)h^{\frac{2}{3}\beta} \big\{\norm{X_0}_{L_2(\Omega,H^\beta)} + 
\norm{\Lambda^{\frac{1}{2}(\beta-1)}Q^{\frac12}}_{\mathrm{HS}}\big\}.
\end{equation*}
\item If $u_{h,0} = \mathcal{R}_hu_0$, $v_{h,0} = \mathcal{P}_hv_0$ and $\beta\in[1,4]$, then
\begin{equation*}
\norm{u_{h,2}(t)-u_2(t)}_{L_2(\Omega,\dot{H}^0)} \leq C(t)h^{\frac23(\beta-1)} \big\{\norm{X_0}_{L_2(\Omega,H^\beta)} 
+ \norm{\Lambda^{\frac12(\beta-1)}Q^{\frac12}}_{\mathrm{HS}}\}.
\end{equation*}
\end{itemize}
\end{theorem}

\begin{remark}\label{order-remark} 
{\it Note that the order of convergence in the position, $\frac23\beta$, is lower than
  the order of regularity, $\beta$, in \eqref{weaksolutionregularity}.
  This is a known feature of the finite element method for the wave
  equation, see \cite{kls}. The upper limits for $\beta$ are only
  dictated by the fact that the maximal order for piecewise linear
  approximation is $2$; higher regularity will not yield higher rate
  of convergence unless higher order finite elements are used, which
  can be done of course, see \cite{kls}. Similarly, it is shown in 
\cite[Theorem 4.1]{KLLweaktime} that the order of convergence 
of implicit one-step temporal approximations is $\bigo{k^{  \min( \beta \frac{p}{p+1} , 1) }}$, 
where $k$ is the steplength and $p$ is the order of the method. 
Thus, $p=1$ and $p=2$ for the backward Euler-Maruyama and Crank-Nicolson-Maruyama methods, 
respectively. 
}
\end{remark}

We will also use the following relation between $\Lambda_h$ and $\Lambda$, 
see the proof of Theorem~4.4 in \cite{kll},
\begin{equation}
\norm{\Lambda_h^\alpha\mathcal{P}_h\Lambda^{-\alpha}v}_{L_2(\mathcal{D})}^2\leq \norm{v}_{L_2(\mathcal{D})}^2, 
\ \ \alpha\in[-\tfrac12,1], \ \ v\in \dot{H}^0 = L_2(\mathcal{D}),
\label{lpl}
\end{equation}
where $\mathcal{P}_h$ is the orthogonal projector $\mathcal{P}_h:\dot
H^0\to V_h$. 

Finally, we remark that the assumption that $\mathcal{D}$ is convex and polygonal guarantees that the
triangulations can be exactly fitted to $\partial\mathcal{D}$ and that we have
the elliptic regularity $\lVert v\rVert_{H^2(\mathcal{D})} \le C\lVert \Lambda
v\rVert_{L_2(\mathcal{D})}$ for $v\in D(\Lambda)$.  This simplifies the error analysis of
the finite element method. The assumption of quasi-uniformity
guarantees that we have an inverse inequality and is only used in the
proof of the case $\alpha\in[0,\frac12]$ of \eqref{lpl}. In
particular, it is not needed for the proof of Theorem~\ref{thm:fem} 
and not for the case $\beta=1$ (trace class noise) in the error analysis 
in Theorem~\ref{thm:stm} below.  

\section{A stochastic trigonometric method for the discretisation in time}\label{sect:trigo}
In order to discretise efficiently the 
finite element problem \eqref{femswe1}, 
or \eqref{femswe2}, in time 
one is often interested in using 
explicit methods with large step sizes. 
A standard approach for the 
deterministic case is the leap-frog scheme, but 
unfortunately one has a step-size restriction due to stability issues. 
In the present paper, we will consider a stochastic extension 
of the trigonometric methods. The trigonometric methods are particularly well 
suited for the numerical discretisation of second-order differential 
equations with highly oscillatory solutions, 
see \cite[Chapter XIII]{hlw} for more details. 
As stated above, the exact solution of 
\eqref{femswe2} is found by the variation-of-constants formula and given by \eqref{exactsolfem}. 
We can write $E_h(t)$ as
\begin{equation}
E_h(t) = \begin{bmatrix} C_h(t) & \Lambda_h^{-1/2}S_h(t) \\ -\Lambda_h^{1/2}S_h(t) & C_h(t) \end{bmatrix}
\label{E_h}
\end{equation}
with $C_h(t) = \cos(t\Lambda_h^{1/2})$ and $S_h(t) = \sin(t\Lambda_h^{1/2})$. 
Discretising the stochastic integral in the sense of It\^o, that is, evaluating 
the integrand at the left-end point of the interval, 
leads us to the stochastic trigonometric method. We let $k$ be the 
time step size and $U_1^0 = u_{h,0}$ and $U_2^0 = v_{h,0}$, and 
obtain the numerical scheme 
$U^{n+1}=E_h(k)U^n+E_h(k)\mathcal{P}_hB\Delta W^n$, that is, 
\begin{equation}
\begin{bmatrix} U_1^{n+1} \\ U_2^{n+1} \end{bmatrix} = 
\begin{bmatrix} C_h(k) & \Lambda_h^{-1/2}S_h(k) \\ -\Lambda_h^{1/2}S_h(k) & C_h(k) \end{bmatrix} 
\begin{bmatrix} U_1^n \\ U_{2}^n \end{bmatrix} + \begin{bmatrix} \Lambda_h^{-1/2}S_h(k) \\ C_h(k) \end{bmatrix}
\mathcal{P}_h\Delta W^n,
\label{stm}
\end{equation}
where $\Delta W^n=W(t_{n+1})-W(t_n)$ denotes the Wiener
increments. Here we thus get an approximation $U_j^n\approx
u_{h,j}(t_n)$ of the
exact solution of our finite element problem at the discrete times $t_n=nk$. 

\begin{remark} 
{\it 
The stochastic trigonometric methods \eqref{stm}
    are easily adapted to the numerical time discretisation of
    ($N$-dimensional) systems of nonlinear stochastic differential
    equations of the form $$ \ddot X(t) + \omega^2X(t) = G(X(t)) +\dot
    W(t), $$ where $\omega\in\R^{N\times N}$ is a symmetric positive
    definite matrix and $G(x)\in\R^N$ is a smooth nonlinearity.  In
    this case, one obtains the following explicit numerical scheme
    \cite{cs11} 
\begin{align}\label{trigo-nl} 
      \begin{split}
       \begin{bmatrix}
        X^{n+1}_1\\
        X^{n+1}_2 \end{bmatrix} &= \begin{bmatrix}
        \cos(k\omega) & \omega^{-1}\sin(k\omega)\\
        -\omega \sin(k\omega) &
        \cos(k\omega) \end{bmatrix} \begin{bmatrix}
        X^n_1\\
        X^n_2 \end{bmatrix} 
       \\ &\quad+ \begin{bmatrix}
        \frac{k^2}{2}\Psi G(\Phi X_1^n)\\
        \frac{k}{2}\bigl(\Psi_0G(\Phi X_1^{n})+ \Psi_1G(\Phi
        X^{n+1}_1)\bigr) \end{bmatrix}
        + \begin{bmatrix}
        \omega^{-1}\sin(k\omega)\\
        \cos(k\omega) \end{bmatrix} \Delta W^n, 
    \end{split}
\end{align} where $k$
    denotes the step size and $\Delta
    W^n=W(t_{n+1})-W(t_n)$ the Wiener increments.  Here
    $\Psi=\psi(k\omega)$ and $\Phi=\phi(k\omega)$, where the filter
    functions $\psi, \phi$ are even, real-valued functions with
    $\psi(0)=\phi(0)=1$.  
    Moreover, we have $\Psi_0=\psi_0(k\omega)$,
    $\Psi_1=\psi_1(k\omega)$ with even functions $\psi_0,\psi_1$
    satisfying $\psi_0(0)=\psi_1(0)=1$.  The purpose of these filter
    functions is to attenuate numerical resonances. Moreover, the
    choice of the filter functions may also have a substantial
    influence on the long-time properties of the method, see
    \cite[Chapter XIII]{hlw} for the deterministic case. We will not
    deal with these issues in the present paper.

Numerical experiments for the nonlinear stochastic wave equation
$$ 
\mathrm{d}\dot{u} - \Delta u\, \mathrm{d}t = G(u)\,\mathrm{d}t+\mathrm{d}W 
$$
with a smooth nonlinearity $G$ will be provided in Section~\ref{sect:numexp} 
in order to demonstrate the efficiency of this approach.  
We leave a theoretical investigation of the nonlinear case 
for future works.
}
\end{remark}

For a more detailed derivation of the trigonometric method and
its use for nonlinear wave equations we refer to \cite[Chapter
XIII]{hlw} and \cite{chl} for the deterministic case and to \cite{c10}
and \cite{cs11} for the stochastic case.

In the next section we will see that 
this explicit numerical method permits 
the use of large time step sizes $k$ and 
that the error bounds are independent of the spatial mesh size $h$; 
some of these properties are not shared by, for example, 
the backward Euler-Maruyama scheme, the St\"ormer-Verlet scheme 
or the Crank-Nicolson-Maruyama scheme, as we will see in the numerical 
experiments in Section~\ref{sect:numexp}. 

\section{Mean-square convergence analysis}\label{sect:strong}

In this section, we will derive mean-square error bounds for 
the stochastic trigonometric method \eqref{stm}. Our main result is a global error 
estimate for the time discretisation in Theorem~\ref{thm:stm}. 
Its proof is based on bounds for the local errors in Lemma~\ref{locallemma}. 
Finally, we formulate an error estimate for the full discretisation.  

\begin{theorem}\label{thm:stm}
Consider the numerical discretisation of \eqref{femswe1} 
by the stochastic trigonometric scheme \eqref{stm} 
with temporal step size $k$. 
The global strong errors of the numerical scheme 
satisfy the following estimates: 
\begin{itemize}[leftmargin=*]
\item If $\norm{\Lambda^{(\beta-1)/2}Q^{1/2}}_{\mathrm{HS}}<\infty$
  for some $\beta\ge0$,
then 
\begin{equation*}
\norm{U_1^n-u_{h,1}(t_n)}_{L_2(\Omega,\dot{H}^0)} \leq Ck^{\min\{\beta,1\}}\norm{\Lambda^{(\beta-1)/2}Q^{1/2}}_{\mathrm{HS}}.
\end{equation*}
\item If $\norm{\Lambda^{(\beta-1)/2}Q^{1/2}}_{\mathrm{HS}}<\infty$ for some $\beta\ge1$, 
then
\begin{equation*}
\norm{U_2^n-u_{h,2}(t_n)}_{L_2(\Omega,\dot{H}^0)} \leq Ck^{\min\{\beta-1,1\}}\norm{\Lambda^{(\beta-1)/2}Q^{1/2}}_{\mathrm{HS}}.
\end{equation*}
\end{itemize}
The constant $C=C(T)$ is independent of $h$, $k$, and $n$ with $t_n=nk\leq T$. 
\end{theorem}

For the proof of the above theorem, we will need the following lemma:

\begin{lemma}
Let the local defects $d^n = [d_1^n,d_2^n]^T$ be defined by
\begin{equation*}
\begin{split}
d^n_1 & := \int_{t_n}^{t_{n+1}} \Lambda_h^{-1/2}S_h(t_{n+1}-s)\mathcal{P}_h\,\mathrm{d}W(s) 
- \Lambda_h^{-1/2}S_h(k)\mathcal{P}_h\Delta W^n, \\
d^n_2 & := \int_{t_n}^{t_{n+1}} C_h(t_{n+1}-s)\mathcal{P}_h\,\mathrm{d}W(s) - C_h(k)\mathcal{P}_h\Delta W^n.
\end{split}
\end{equation*}
We have the following estimates: 
\begin{itemize}[leftmargin=*]
\item 
If $\norm{\Lambda^{(\beta-1)/2}Q^{1/2}}_{\mathrm{HS}}<\infty$ for some $\beta\ge0$, then 
\begin{equation*}
\mathbb{E}[\norm{d^n_1}_{L_2(\mathcal{D})}^2] +
\mathbb{E}[\norm{\Lambda_h^{-1/2} d_2^n}_{L_2(\mathcal{D})}^2]
\leq Ck^{\min\{2\beta+1,3\}}\norm{\Lambda^{(\beta-1)/2}Q^{1/2}}_{\mathrm{HS}}^2.
\end{equation*}
\item
If $\norm{\Lambda^{(\beta-1)/2}Q^{1/2}}_{\mathrm{HS}}<\infty$ for some $\beta\ge1$, then
\begin{equation*}
\mathbb{E}[\norm{\Lambda_h^{1/2}d_1^n}_{L_2(\mathcal{D})}^2]
+
\mathbb{E}[\norm{d^n_2}_{L_2(\mathcal{D})}^2] 
\leq Ck^{\min\{2\beta-1,3\}}\norm{\Lambda^{(\beta-1)/2}Q^{1/2}}_{\mathrm{HS}}^2.
\end{equation*}
\end{itemize} 
The constant $C=C(T)$ is independent of $h$, $k$, and $n$ with $t_n=nk\leq T$.
\label{locallemma}
\end{lemma}

\begin{proof} 
We begin by showing, recall that $\dot{H}_h^0=V_h$ with norm 
$\norm{\cdot}_{h,0}=\norm{\cdot}_{L_2(\mathcal{D})}$, 
\begin{equation}
\norm{(S_h(t)-S_h(s))\Lambda_h^{-\beta/2}}_{\mathcal{L}(\dot{H}_h^0)}
\leq C|t-s|^{\beta}, \quad\beta\in[0,1].
\label{sinest}
\end{equation}
For $\beta=0$ and $v_h\in V_h$ we use the triangle inequality 
and the boundedness of $S_h(t)$: 
$$ 
\norm{(S_h(t)-S_h(s))v_h}_{L_2(\mathcal{D})} \leq 
2\norm{v_h}_{L_2(\mathcal{D})}=2\norm{v_h}_{h,0}.
$$
For $\beta=1$ and $v_h\in V_h$ we use 
the fact that 
$$
(S_h(t)-S_h(s))v_h
=\int_s^t\mathrm{D}_rS_h(r)v_h\,\mathrm{d}r
=\int_s^tC_h(r)\Lambda_h^{1/2}v_h\,\mathrm{d}r
$$
and hence 
\begin{align*}
\norm{(S_h(t)-S_h(s))v_h}_{L_2(\mathcal{D})} 
\le |t-s| \norm{\Lambda_h^{1/2} v_h}_{L_2(\mathcal{D})} 
= |t-s| \norm{v_h}_{h,1}.
\end{align*}
A well-known interpolation argument, see e.g.\ the proof of Theorem~3.5 in \cite{t06}, 
then yields 
\begin{equation*}
\norm{(S_h(t)-S_h(s))v_h}_{L_2(\mathcal{D})} 
\leq C|t-s|^{\beta} \norm{v_h}_{h,\beta},\quad v_h\in V_h,\ \beta\in[0,1],
\end{equation*}
which is \eqref{sinest}.  
 
We now consider $d_1^n$ with $\beta\in[0,1]$. 
By It\^o's isometry \eqref{ito} and \eqref{sinest} we have  
\begin{align*}
\mathbb{E}[\norm{d^n_1}_{L_2(\mathcal{D})}^2] 
& = 
\mathbb{E}\Big[\Big\lVert\int_{t_n}^{t_{n+1}}\Lambda_h^{-1/2}(S_h(t_{n+1}-s)-S_h(k))
\mathcal{P}_h\,\mathrm{d}W(s)\Big\rVert_{L_2(\mathcal{D})}^2\Big] \\
& = 
\int_{0}^{k}\norm{\Lambda_h^{-1/2}(S_h(s)-S_h(k))\mathcal{P}_hQ^{1/2}}_{\mathrm{HS}}^2\,\mathrm{d}s \\
& \leq 
\int_{0}^{k}\norm{(S_h(s)-S_h(k)) \Lambda_h^{-\beta/2}}_{\mathcal{L}(\dot{H}_h^0)}^2
\,\mathrm{d}s \,
\norm{\Lambda_h^{(\beta-1)/2}\mathcal{P}_hQ^{1/2}}_{\mathrm{HS}}^2 \\
&\le 
Ck^{2\beta+1}
\norm{\Lambda_h^{(\beta-1)/2}\mathcal{P}_hQ^{1/2}}_{\mathrm{HS}}^2 .
\end{align*}
Using also \eqref{lpl} with $\alpha=(\beta-1)/2\in[-\frac12,0]$ we obtain
\begin{align*}
\norm{\Lambda_h^{(\beta-1)/2}\mathcal{P}_hQ^{1/2}}_{\mathrm{HS}}
& = 
\norm{\Lambda_h^{(\beta-1)/2}\mathcal{P}_h\Lambda^{-(\beta-1)/2}\Lambda^{(\beta-1)/2}Q^{1/2}}_{\mathrm{HS}}\\
&\le 
\norm{\Lambda_h^{(\beta-1)/2}\mathcal{P}_h\Lambda^{-(\beta-1)/2}}_{\mathcal{L}(\dot{H}^0)}
\norm{\Lambda^{(\beta-1)/2}Q^{1/2}}_{\mathrm{HS}}\\
& \leq C\norm{\Lambda^{(\beta-1)/2}Q^{1/2}}_{\mathrm{HS}}.   
\end{align*}
This proves 
\begin{align*}
\mathbb{E}[\norm{d^n_1}_{L_2(\mathcal{D})}^2] 
\le  
Ck^{2\beta+1}\norm{\Lambda^{(\beta-1)/2}Q^{1/2}}^2_{\mathrm{HS}},    
\end{align*} 
which is the desired bound when $\beta\in[0,1]$. When $\beta\ge1$, we
simply observe that
$\norm{\Lambda^{-(\beta-1)/2}}_{\mathcal{L}(\dot{H}^0)}\le C $, so that by
the already proven case 
\begin{align*}
\mathbb{E}[\norm{d^n_1}_{L_2(\mathcal{D})}^2] 
&\,\le \int_0^k\norm{\Lambda_h^{-1/2}(S_h(s)-S_h(k))}_{\mathcal{L}(\dot{H}_h^0)}^2\,\mathrm{d}s\,
\norm{\mathcal{P}_hQ^{1/2}}_{\mathrm{HS}}^2 \\
&\le C\int_0^k(s-k)^2\,\mathrm{d}s\,\norm{\mathcal{P}_hQ^{1/2}}_{\mathrm{HS}}^2
\le Ck^{3} \norm{{Q^{1/2}}}_{\mathrm{HS}}^2\\ 
&\le 
Ck^{3}
\norm{\Lambda^{ (\beta-1)/2}Q^{1/2}}_{\mathrm{HS}}^2    
\norm{\Lambda^{-(\beta-1)/2}}^2_{\mathcal{L}(\dot{H}^0)} \\
& \le 
Ck^{3}
\norm{\Lambda^{(\beta-1)/2}Q^{1/2}}_{\mathrm{HS}}^2 .
\end{align*} 
This is the desired result for $\beta\ge1$.  

Similarly we find for the second component $d^n_2$ with $\beta\in[1,2]$:
\begin{align*}
\mathbb{E}[\norm{d^n_2}_{L_2(\mathcal{D})}^2] 
& \leq
\int_{0}^{k} \norm{(C_h(s)-C_h(k)) \Lambda_h^{-(\beta-1)/2}}_{\mathcal{L}(\dot{H}_h^0)}^2
\,\mathrm{d}s \,
\norm{\Lambda_h^{(\beta-1)/2}\mathcal{P}_hQ^{1/2}}_{\mathrm{HS}}^2 , 
\end{align*}
where, similar to \eqref{sinest}, 
\begin{align*}
\norm{(C_h(t)-C_h(s))\Lambda^{-(\beta-1)/2}_h}_{\mathcal{L}(\dot{H}_h^0)} 
\leq C|t-s|^{\beta-1}, \quad \beta\in[1,2]. 
\end{align*}
Hence, using also \eqref{lpl} now with $\alpha=(\beta-1)/2\in[0,\frac12]$, we obtain 
\begin{align*}
\mathbb{E}[\norm{d^n_2}_{L_2(\mathcal{D})}^2] 
\leq C k^{2\beta-1} \norm{\Lambda^{(\beta-1)/2}Q^{1/2}}_{\mathrm{HS}}^2
\end{align*}
for $\beta\in[1,2]$. For $\beta \ge2$ the defect is of the order $k^3$. 

The bounds for
  $\mathbb{E}[\norm{\Lambda_h^{1/2}d_1^n}_{L_2(\mathcal{D})}^2]$ and
  $\mathbb{E}[\norm{\Lambda_h^{-1/2} d_2^n}_{L_2(\mathcal{D})}^2]$
are proved in the same way.
\end{proof}

We now turn to the proof of our main result on the 
strong convergence of the numerical method \eqref{stm}.
\begin{proof}[Proof of Theorem \ref{thm:stm}]
We define $F_j^n := U_j^n-u_{h,j}(t_n)$, $j=1,2$, 
and $F^n = [F_1^n,F_2^n]^T$. First of all we remark that
\begin{equation*}
\norm{U_1^n-u_{h,1}(t_n)}_{L_2(\Omega,\dot{H}^0)}^2 = \norm{F_1^n}_{L_2(\Omega,\dot{H}^0)}^2 = 
\mathbb{E}\big[\norm{F_1^n}_{L_2(\mathcal{D})}^2\big].
\end{equation*}
Substituting the exact solution $X_h = [u_{h,1},u_{h,2}]^T$ of \eqref{femswe2} 
into the numerical scheme \eqref{stm}, we obtain
\begin{equation*}
X_h(t_{n+1}) = E_h(k)X_h(t_n) + E_h(k)\mathcal{P}_hB\Delta W^n + d^n
\end{equation*}
with the defects $d^n := [d^n_1,d^n_2]^T$ defined in Lemma~\ref{locallemma} 
and $E_h(t)$ defined in \eqref{E_h}. We thus obtain the 
following formula for the error $F^{n+1}$:
\begin{equation*}
F^{n+1} = E_h(k)F^n + d^n = E_h(t_{n+1})F^0 + \sum_{j=0}^n E_h(t_{n-j})d^j = \sum_{j=0}^n E_h(t_{n-j})d^j,
\end{equation*}
since $F^0=0$.
Taking expectations gives us for the first component
\begin{align*}
\mathbb{E}\big[\norm{F^{n}_1}_{L_2(\mathcal{D})}^2\big] 
&= \mathbb{E}\Big[\Big\lVert\sum_{j=0}^{n-1} \bigl(C_h(t_{n-1-j})d_1^j+\Lambda_h^{-1/2}S_h(t_{n-1-j})d_2^j\bigr)\Big\rVert_{L_2(\mathcal{D})}^2\Big] \\
& =\mathbb{E}\bigg[\biggl( \sum_{j=0}^{n-1} C_h(t_{n-1-j})d_1^j,\sum_{i=0}^{n-1} C_h(t_{n-1-i})d_1^i\biggr) \\
&\quad +\biggl( \sum_{j=0}^{n-1} C_h(t_{n-1-j})d_1^j,\sum_{i=0}^{n-1} \Lambda_h^{-1/2}S_h(t_{n-1-i})d_2^i\biggr) \\
&\quad +\biggl( \sum_{j=0}^{n-1} \Lambda_h^{-1/2}S_h(t_{n-1-j})d_2^j,\sum_{i=0}^{n-1} C_h(t_{n-1-i})d_1^i\biggr) \\
&\quad +\biggl( \sum_{j=0}^{n-1} \Lambda_h^{-1/2}S_h(t_{n-1-j})d_2^j,\sum_{i=0}^{n-1} \Lambda_h^{-1/2}S_h(t_{n-1-i})d_2^i\biggr)\bigg] .
\end{align*}
Here we use the independence of $d_{1,2}^i$ and 
$d_{1,2}^j$ with $i,j=0,\ldots,n-1$ for $i\neq j$ to get
\begin{align*}
\mathbb{E}\big[\norm{F^{n}_1}_{L_2(\mathcal{D})}^2\big] 
&
=\mathbb{E}\Big[\sum_{j=0}^{n-1}(C_h(t_{n-1-j})d_1^j,C_h(t_{n-1-j})d_1^j)\\
& \quad + \sum_{j=0}^{n-1} (C_h(t_{n-1-j})d_1^j,\Lambda_h^{-1/2}S_h(t_{n-1-j})d_2^j) \\
&\quad +\sum_{j=0}^{n-1} (\Lambda_h^{-1/2}S_h(t_{n-1-j})d_2^j,C_h(t_{n-1-j})d_1^j) \\
&\quad +\sum_{j=0}^{n-1} (\Lambda_h^{-1/2}S_h(t_{n-1-j})d_2^j,\Lambda_h^{-1/2}S_h(t_{n-1-j})d_2^j)\Big] \\
& = \sum_{j=0}^{n-1} \mathbb{E}\Big[\norm{C_h(t_{n-1-j})d_1^j+\Lambda_h^{-1/2}S_h(t_{n-1-j})d_2^j}_{L_2(\mathcal{D})}^2\Big] \\
& \leq 2\sum_{j=0}^{n-1} \Bigl(\mathbb{E}\big[\norm{d_1^j}_{L_2(\mathcal{D})}^2\big] + 
\mathbb{E}\big[\norm{\Lambda_h^{-1/2} d_2^j}_{L_2(\mathcal{D})}^2\big]\Bigr).
\end{align*} 
Now we can apply Lemma~\ref{locallemma} for the estimates of the defects $d_1^j$ and $d_2^j$ and get
\begin{align*}
\mathbb{E}\big[\norm{F_1^{n}}_{L_2(\mathcal{D})}^2\big] & \leq  \displaystyle
C \sum_{j=0}^n k^{\min\{2\beta+1,3\}}\norm{\Lambda^{(\beta-1)/2}Q^{1/2}}_{\mathrm{HS}}^2 \\
& \leq \displaystyle C(T) k^{\min\{2\beta,2\}}\norm{\Lambda^{(\beta-1)/2}Q^{1/2}}_{\mathrm{HS}}^2.
\end{align*}
Therefore we obtain 
\begin{equation*}
\norm{U_1^n-u_{h,1}(t_n)}_{L_2(\Omega,\dot{H}^0)} = \sqrt{\mathbb{E}\big[\norm{F_1^n}_{L_2(\mathcal{D})}^2\big]} 
\leq Ck^{\min\{\beta,1\}}\norm{\Lambda^{(\beta-1)/2}Q^{1/2}}_{\mathrm{HS}}
\end{equation*}
for $\beta\ge0$. 

For the second component of $F^n$ we obtain 
\begin{align*}
\mathbb{E}\big[\norm{F^{n}_2}_{L_2(\mathcal{D})}^2\big] & = 
\mathbb{E}\Big[\Big\lVert\sum_{j=0}^{n-1} \bigl(-\Lambda_h^{1/2}S_h(t_{n-1-j})d_1^j+C_h(t_{n-1-j})d_2^j\bigr)\Big\rVert_{L_2(\mathcal{D})}^2\Big] \\
& = \sum_{j=0}^{n-1} \mathbb{E}\big[\norm{-\Lambda_h^{1/2}S_h(t_{n-1-j})d_1^j+C_h(t_{n-1-j})d_2^j}_{L_2(\mathcal{D})}^2\big] \\
& \leq C\sum_{j=0}^{n-1} \bigl(\norm{\Lambda_h^{1/2}d_1^j}_{L_2(\mathcal{D})}^2 + \norm{d_2^j}_{L_2(\mathcal{D})}^2\bigr).
\end{align*}
Thus we get with Lemma~\ref{locallemma}, if $\beta\ge1$: 
\begin{align*}
\mathbb{E}\big[\norm{F_2^{n}}_{L_2(\mathcal{D})}^2\big] & \leq \displaystyle 
C \sum_{j=0}^n k^{\min\{2\beta-1,3\}}\norm{\Lambda^{(\beta-1)/2}Q^{1/2}}_{\mathrm{HS}}^2 \\
& \leq \displaystyle C k^{\min\{2\beta-2,2\}}\norm{\Lambda^{(\beta-1)/2}Q^{1/2}}_{\mathrm{HS}}^2
\end{align*}
and
\begin{equation*}
\norm{U_2^n-u_{h,2}(t_n)}_{L_2(\Omega,\dot{H}^0)} =
\sqrt{\mathbb{E}\big[\norm{F_2^n}_{L_2(\mathcal{D})}^2\big]}
\leq 
Ck^{\min\{\beta-1,1\}}\norm{\Lambda^{(\beta-1)/2}Q^{1/2}}_{\mathrm{HS}}.
\end{equation*}
\end{proof}

We can now collect the convergence results for 
the space discretisation and for the time discretisation. 
This gives us the following theorem.
\begin{theorem}
Consider the numerical solution of \eqref{swe} by 
the finite element method in space 
with a maximal mesh size $h$ 
and the numerical scheme \eqref{stm} with 
a time step size $k$ on the time interval $[0,T]$. 
Let us denote the 
discrete time by $t_n=nk$. Let $X_0=[u_0,v_0]^T$ and let $X = [u_1,u_2]^T$ and $X_h=[u_{h,1},u_{h,2}]^T$ 
be given by \eqref{exactsol} and \eqref{exactsolfem}, respectively. If $\norm{X_0}_{L_2(\Omega,H^\beta)}<\infty$, 
the following estimates hold for $t\geq 0$, where $C(t)$ is an increasing function of the time $t$. 
\begin{itemize}[leftmargin=*]
\item If $u_{h,0} = \mathcal{P}_hu_0$, $v_{h,0} = \mathcal{P}_hv_0$ and 
if $\norm{\Lambda^{(\beta-1)/2}Q^{1/2}}_{\mathrm{HS}}<\infty$ for some $\beta\in[0,3]$, then
\begin{equation*}
\norm{U_1^n-u_1(t_n)}_{L_2(\Omega,\dot{H}^0)} \leq C(T)\Big(h^{2\beta/3}+k^{\min\{\beta,1\}}\Big)\norm{\Lambda^{(\beta-1)/2}Q^{1/2}}_{\mathrm{HS}}.
\end{equation*}
\item If $u_{h,0} = \mathcal{R}_hu_0$, $v_{h,0} = \mathcal{P}_hv_0$ and 
if $\norm{\Lambda^{(\beta-1)/2}Q^{1/2}}_{\mathrm{HS}}<\infty$ for some $\beta\in[1,4]$, then
\begin{equation*}
\norm{U_2^n-u_2(t_n)}_{L_2(\Omega,\dot{H}^0)} \leq C(T)\Big(h^{2(\beta-1)/3}+k^{\min\{\beta-1,1\}}\Big)\norm{\Lambda^{(\beta-1)/2}Q^{1/2}}_{\mathrm{HS}}.
\end{equation*}
\end{itemize}
\end{theorem}

\begin{proof}
This follows from Theorems~\ref{thm:fem} and~\ref{thm:stm} 
by the triangle inequality.  
\end{proof}

\section{A trace formula for the numerical solution}\label{sect:trace}

In this section, we look at a geometric property of the exact solution of the wave equation. 
It is known that, in the deterministic setting, the linear wave equation 
is a Hamiltonian partial differential equation, 
wherein the total energy (or Hamiltonian) of the problem is conserved 
for all times. However, in the stochastic case considered here, 
the expected value of the energy grows linearly with the time $t$. 
This is stated in the next theorem for the semidiscretisation of our 
linear stochastic wave equation \eqref{swe}. 
For a nonlinear version of this so-called 
trace formula we refer to \cite{sch08}.

\begin{theorem}
Consider the numerical solution of \eqref{swe} by 
the finite element method in space 
with a maximal mesh size $h$. 
Let $X_h=[u_{h,1},u_{h,2}]^T$ be given by \eqref{exactsolfem}. 
The expected value of the energy of the exact solution of the 
semidiscrete problem \eqref{femswe1} with initial 
values $X_h(0)=[u_{h,0},v_{h,0}]^T\in L_2(\Omega,V_h)$ satisfies: 
\begin{align*}
\mathbb{E}\Big[\frac{1}{2}\big(\norm{\Lambda_h^{1/2} u_{h,1}(t)}_{L_2(\mathcal{D})}^2 + 
\norm{u_{h,2}(t)}_{L_2(\mathcal{D})}^2\big)\Big] 
& = \mathbb{E}\Big[\frac{1}{2}\big(\norm{\Lambda_h^{1/2} u_{h,0}}_{L_2(\mathcal{D})}^2 + 
\norm{v_{h,0}}_{L_2(\mathcal{D})}^2\big)\Big] \\
&\quad +  \frac12t\mathrm{Tr}(\mathcal{P}_hQ\mathcal{P}_h)
\end{align*}
\label{driftexact}
for all times $t\ge 0$. 
\end{theorem}

\begin{proof}
We recall that the solution of \eqref{femswe1}, $X_h(t)=[u_{h,1}(t),u_{h,2}(t)]^T$, 
with initial values $X_h(0)=[u_{h,0},v_{h,0}]^T$ can be written as 
$$ 
X_h(t) = E_h(t)X_{h}(0) + \int_0^tE_h(t-s)\mathcal{P}_hB\,\mathrm{d}W(s). 
$$
Therefore we get for the first summand of the energy, i.\,e., the potential energy, 
\begin{align*}
\mathbb{E}\Big[\norm{\Lambda_h^{1/2} u_{h,1}(t)}_{L_2(\mathcal{D})}^2\Big] & = 
\displaystyle \mathbb{E}\Big[\Norm{\Lambda_h^{1/2}C_h(t)u_{h,0} + S_h(t)v_{h,0} + 
\int_0^tS_h(t-s)\mathcal{P}_h\,\mathrm{d}W(s)}_{L_2(\mathcal{D})}^2\Big] \\
& = \displaystyle \mathbb{E}\Big[\norm{\Lambda_h^{1/2}C_h(t)u_{h,0}}_{L_2(\mathcal{D})}^2 + \norm{S_h(t)v_{h,0}}_{L_2(\mathcal{D})}^2 \\
& \quad+ \displaystyle \Norm{\int_0^tS_h(t-s)\mathcal{P}_h\,\mathrm{d}W(s)}_{L_2(\mathcal{D})}^2 + 2\big(\Lambda_h^{1/2}C_h(t)u_{h,0},S_h(t)v_{h,0}\big) \\
& \quad+ \displaystyle 2\Big(\Lambda_h^{1/2}C_h(t)u_{h,0},\int_0^tS_h(t-s)\mathcal{P}_h\,\mathrm{d}W(s)\Big) \\
& \quad+ \displaystyle 2\Big(S_h(t)v_{h,0},\int_0^tS_h(t-s)\mathcal{P}_h\,\mathrm{d}W(s)\Big)\Big] \\
& = \displaystyle \mathbb{E}\Big[\norm{\Lambda_h^{1/2}C_h(t)u_{h,0}}_{L_2(\mathcal{D})}^2 + \norm{S_h(t)v_{h,0}}_{L_2(\mathcal{D})}^2 \\
& \quad+ \displaystyle \Norm{\int_0^tS_h(t-s)\mathcal{P}_h\,\mathrm{d}W(s)}_{L_2(\mathcal{D})}^2 + 
2\big(\Lambda_h^{1/2}C_h(t)u_{h,0},S_h(t)v_{h,0}\big)\Big]
\end{align*}
using the fact that the above It\^o integrals are normally distributed with mean $0$. 

For the second summand we obtain
\begin{align*}
\mathbb{E}\Big[\norm{u_{h,2}(t)}_{L_2(\mathcal{D})}^2\Big] 
& = \displaystyle \mathbb{E}\Big[\norm{\Lambda_h^{1/2}S_h(t)u_{h,0}}_{L_2(\mathcal{D})}^2 + \norm{C_h(t)v_{h,0}}_{L_2(\mathcal{D})}^2 \\
&\quad +\displaystyle \Norm{\int_0^tC_h(t-s)\mathcal{P}_h\,\mathrm{d}W(s)}_{L_2(\mathcal{D})}^2 - 
2\big(\Lambda_h^{1/2}C_h(t)u_{h,0},S_h(t)v_{h,0}\big)\Big].
\end{align*}
Now, we use It\^o's isometry to compute, for example, 
\begin{equation*}
\mathbb{E}\Big[\Norm{\int_0^tS_h(t-s)\mathcal{P}_h\,\mathrm{d}W(s)}_{L_2(\mathcal{D})}^2\Big] = 
\int_0^t\norm{S_h(t-s)\mathcal{P}_hQ^{1/2}}_{\mathrm{HS}}^2\,\mathrm{d}s.
\end{equation*}
Then, combining these expressions and using a trigonometric identity leads to the statement of the theorem:
\begin{align*}
\mathbb{E}\Big[\frac{1}{2}\big(\norm{\Lambda_h^{1/2} u_{h,1}(t)}_{L_2(\mathcal{D})}^2 + 
\norm{u_{h,2}(t)}_{L_2(\mathcal{D})}^2\big)\Big] & = 
\mathbb{E}\Big[\frac{1}{2}\big(\norm{\Lambda_h^{1/2} u_{h,0}}_{L_2(\mathcal{D})}^2 
+ \norm{u_{h,0}}_{L_2(\mathcal{D})}^2\big)\Big] \\
&\quad +  \frac{1}{2}t\norm{\mathcal{P}_hQ^{1/2}}_{\mathrm{HS}}^2 \\
& =  \frac{1}{2}\big(\norm{\Lambda_h^{1/2} u_{h,0}}_{L_2(\mathcal{D})}^2 + 
\norm{u_{h,0}}_{L_2(\mathcal{D})}^2\big) \\
&\quad +  \frac{1}{2}t\mathrm{Tr}(\mathcal{P}_hQ\mathcal{P}_h).
\end{align*}
The last equality follows from the definitions of the HS-norm, 
of the operator $Q$ and of the projector $\mathcal{P}_h$: 
\begin{align*}
\norm{\mathcal{P}_hQ^{1/2}}_{\mathrm{HS}}^2 & = \displaystyle 
\mathrm{Tr}( (\mathcal{P}_hQ^{1/2})(\mathcal{P}_hQ^{1/2})^* )
=\mathrm{Tr}( \mathcal{P}_hQ\mathcal{P}_h ).
\end{align*}
This concludes the proof.
\end{proof}
\begin{remark} 
{\it 
We would like to point out, that an alternative proof of the above result can be obtained 
using It\^o's formula, see for example \cite[Theorem~4.17]{dz92}, to the 
function 
$$
F(U_h)=\frac{1}{2}\big(\norm{\Lambda_h^{1/2} U_{h,1}}_{L_2(\mathcal{D})}^2 + 
\norm{U_{h,2}}_{L_2(\mathcal{D})}^2\big).
$$
}
\end{remark}
We are now able to show that the numerical solution given by our 
stochastic trigonometric scheme preserves this geometric property of the exact solution 
of the finite element problem \eqref{femswe1}. 
\begin{theorem}\label{driftnum}
Under the assumptions of Theorem~\ref{driftexact}, the numerical solution of \eqref{femswe1} by the stochastic trigonometric 
method \eqref{stm} with a step size $k$ preserves the linear drift of the expected value of the energy, i.\,e.,
\begin{align*}
\mathbb{E}\Big[\frac{1}{2}\big(\norm{\Lambda_h^{1/2} U_1^n}_{L_2(\mathcal{D})}^2 + 
\norm{U_2^n}_{L_2(\mathcal{D})}^2\big)\Big] & =  
\mathbb{E}\Big[\frac{1}{2}\big(\norm{\Lambda_h^{1/2} u_{h,0}}_{L_2(\mathcal{D})}^2 + 
\norm{v_{h,0}}_{L_2(\mathcal{D})}^2\big)\Big] \\
& \quad+ \frac12t_n\mathrm{Tr}(\mathcal{P}_hQ\mathcal{P}_h)
\end{align*}
for all times $t_n=nk\ge 0$.
\end{theorem}

\begin{proof}
The stochastic part of the method can be written as an 
It\^o integral and we obtain due to the It\^o isometry
\begin{align*}
\mathbb{E}\Big[\norm{S_h(k)\mathcal{P}_h\Delta W^{n-1}}_{L_2(\mathcal{D})}^2\Big] 
& =  \mathbb{E}\Big[\Norm{\int_{t_{n-1}}^{t_n}S_h(k)\mathcal{P}_h\,\mathrm{d}W(s)}_{L_2(\mathcal{D})}^2\Big] \\
& = \displaystyle\int_{t_{n-1}}^{t_n}\norm{S_h(k)\mathcal{P}_hQ^{1/2}}_{\mathrm{HS}}^2\,\mathrm{d}s.
\end{align*}
Similarly to the proof of Theorem~\ref{driftexact} we thus get
\begin{align*}
\mathbb{E}\Big[\frac{1}{2}\big(\norm{\Lambda_h^{1/2} U_1^n\big}_{L_2(\mathcal{D})}^2 + 
\norm{U_2^n}_{L_2(\mathcal{D})}^2\big)\Big] 
& = \mathbb{E}\Big[\frac{1}{2}\big(\norm{\Lambda_h^{1/2} U_1^{n-1}}_{L_2(\mathcal{D})}^2 + 
\norm{U_2^{n-1}}_{L_2(\mathcal{D})}^2\big)\Big] \\
& \quad+  \frac{k}{2}\mathrm{Tr}(\mathcal{P}_hQ\mathcal{P}_h).
\end{align*}
A recursion now concludes the proof.
\end{proof} 

To conclude this section, we would like to remark that already 
for stochastic ordinary differential equations, the growth rate of the expected 
energy along the numerical solutions given by the forward (or backward) 
Euler-Maruyama scheme and the midpoint rule, see \cite{c10} and references 
therein, is not correct. Indeed, for the forward Euler-Maruyama scheme, one has 
an exponential drift in the expected value of the energy. 

\section{Numerical examples}\label{sect:numexp}

Let us consider the example given in \cite{kls}:
\begin{equation}\label{example}
  \begin{aligned}
&\mathrm{d}\dot{u} - \Delta u\, \mathrm{d}t = \mathrm{d}W, && \quad (x,t) \in \ (0,1)\times(0,1), \\
&u(0,t)=u(1,t) = 0, &&\quad t\in(0,1), \\
&u(x,0) = \cos(\pi(x-1/2)), \ \dot u(x,0)=0, &&\quad x \in(0,1).
\end{aligned}
\end{equation}
The solution of this stochastic partial differential equation will now be 
numerically approximated with a finite element 
method in space and the stochastic trigonometric method \eqref{stm} in time. 
For the below numerical experiments, we will consider two kinds of noise: 
a space-time white noise with covariance operator $Q=I$ and a correlated one. 
For correlated noise we choose $Q=\Lambda^{-s}$ with $s\in \mathbb{R}$ and 
recall the relation $\beta<1+s-d/2$, where $d=1$ is the dimension 
of the problem, see the discussion after \eqref{Qassumption}. 

Before we start with our numerical experiments, 
let us briefly explain how we approximate the noise 
present in the above stochastic partial differential equation. 
From the Fourier expansion \eqref{fnoise}, we have for all $\chi\in V_h$:
$$
(\mathcal{P}_h\Delta W^n,\chi)_{L_2(\mathcal{D})}=
\sum_{j=1}^\infty\gamma_j^{1/2}\Delta\beta_j^n(e_j,\chi)_{L_2(\mathcal{D})},
$$
where $\{\gamma_j,e_j\}_{j=1}^\infty$ are the eigenpairs of 
the covariance operator $Q$ with 
orthonormal eigenvectors $\{e_j\}_{j=1}^\infty$, and 
$\{\beta_j\}_{j=1}^\infty$ are mutually 
independent standard real-valued Brownian motions with Gaussian 
increments $\Delta\beta_j^n=\beta_j(t_n)-\beta_j(t_{n-1})\sim\sqrt{k}\mathcal{N}(0,1)$. 
As explained in \cite{kls}, under some assumptions on the triangulation 
and the operator $Q$, one can approximate the above expansion 
with 
$$
(\mathcal{P}_h\Delta W^n,\chi)_{L_2(\mathcal{D})}
\approx\sum_{j=1}^J\gamma_j^{1/2}\Delta\beta_j^n(e_j,\chi)_{L_2(\mathcal{D})},
$$
with an integer $J\geq N_h$, where $N_h=\text{dim}(V_h)$, 
while retaining the convergence rate, to obtain the semidiscrete solution, see \eqref{exactsolfem}, 
$$
X_h^J(t)=E_h(t)X_{h,0}+\sum_{j=1}^J\gamma_j^{1/2}\int_0^tE_h(t-s)\mathcal{P}_hBe_j\,
\mathrm d\beta_j(s).
$$
Figure~\ref{fig:errorh} confirms the results on the spatial 
discretisation of our linear stochastic wave equation 
stated in Theorem~\ref{thm:fem}. The spatial errors in the first component 
of our problem are displayed for various values of the parameter $s$. 
On the one hand we consider a space-time white noise with $Q=I$, and hence $\beta<1/2$, 
and on the other hand, different correlated noises with $Q=\Lambda^{-s}$, 
i.\,e., $\beta < 1/2+s$. The corresponding convergence rates are observed. 
Here, we simulate the exact solution with the numerical one  
using a very small step size, i.\,e.,  
$k_{\mathrm{exact}}=h_{\mathrm{exact}}=2^{-8}$. 
The expected values are approximated by computing 
averages over $M=100$ samples. 
All the numerical experiments were performed 
in Matlab using specially designed software 
and the random numbers were generated with 
the command \verb+randn('state',100)+. 

\begin{figure}
\begin{center}
\includegraphics*[height=7cm,keepaspectratio]
{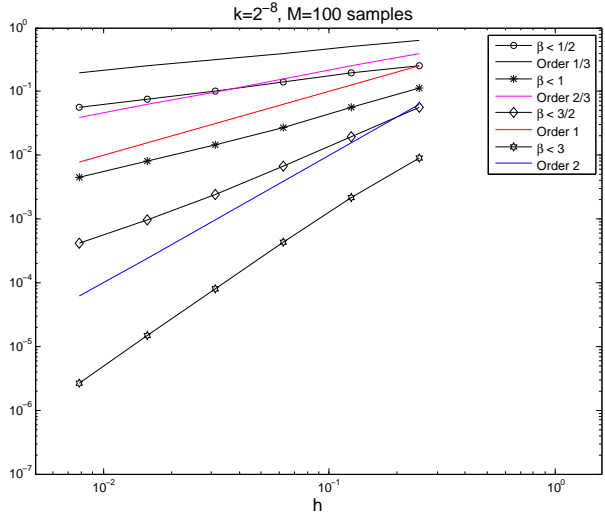}
\caption{Spatial errors: The $L_2$-error in the first component 
decreases with order $h^{\frac23\beta}$.}
\label{fig:errorh}
\end{center}
\end{figure}

We are now interested in the time-discretisation of the above 
stochastic wave equation for various spatial meshes. 
Figure~\ref{fig:errork} displays 
the strong error at time $t=1$ 
in the first component of the solution 
for space-time white noise with $s=0$ and 
for correlated noise with $s=1/2$, respectively. 
One observes the order of convergence stated in Theorem~\ref{thm:stm} 
and the fact that these errors are independent of the spatial discretisation. 
Again, the exact solution is approximated by the stochastic trigonometric method 
with a very small step size $k_{\mathrm{exact}}=2^{-6}$. 
We use $h_{\mathrm{exact}}=2^{-9}, 2^{-10}$, resp., $2^{-11}$ 
for the spatial discretisations. Again $M=100$ samples are used 
for the approximation of the expected values. 

\begin{figure}
\begin{center}
\includegraphics*[height=7cm,keepaspectratio]
{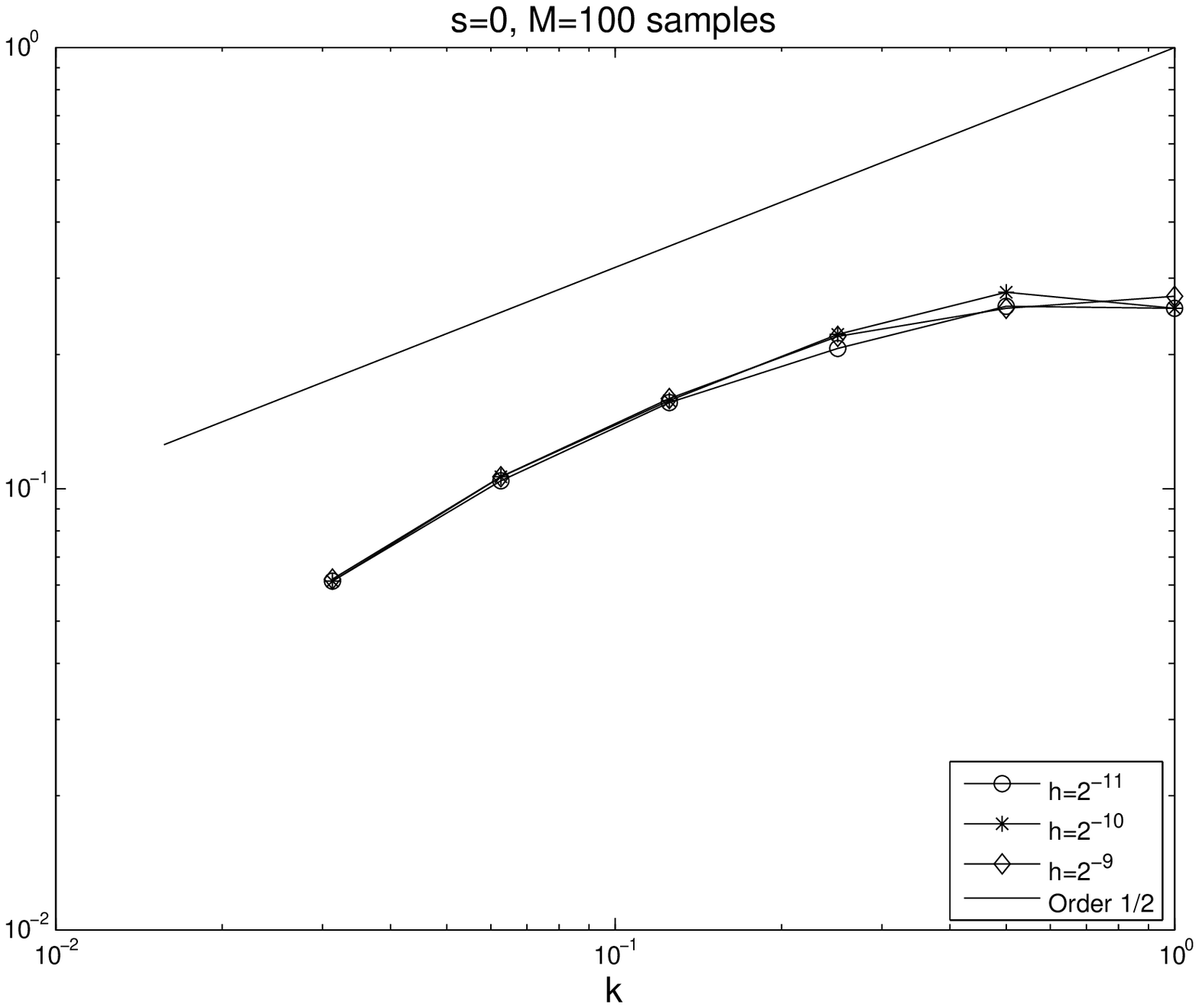}
\includegraphics*[height=7cm,keepaspectratio]
{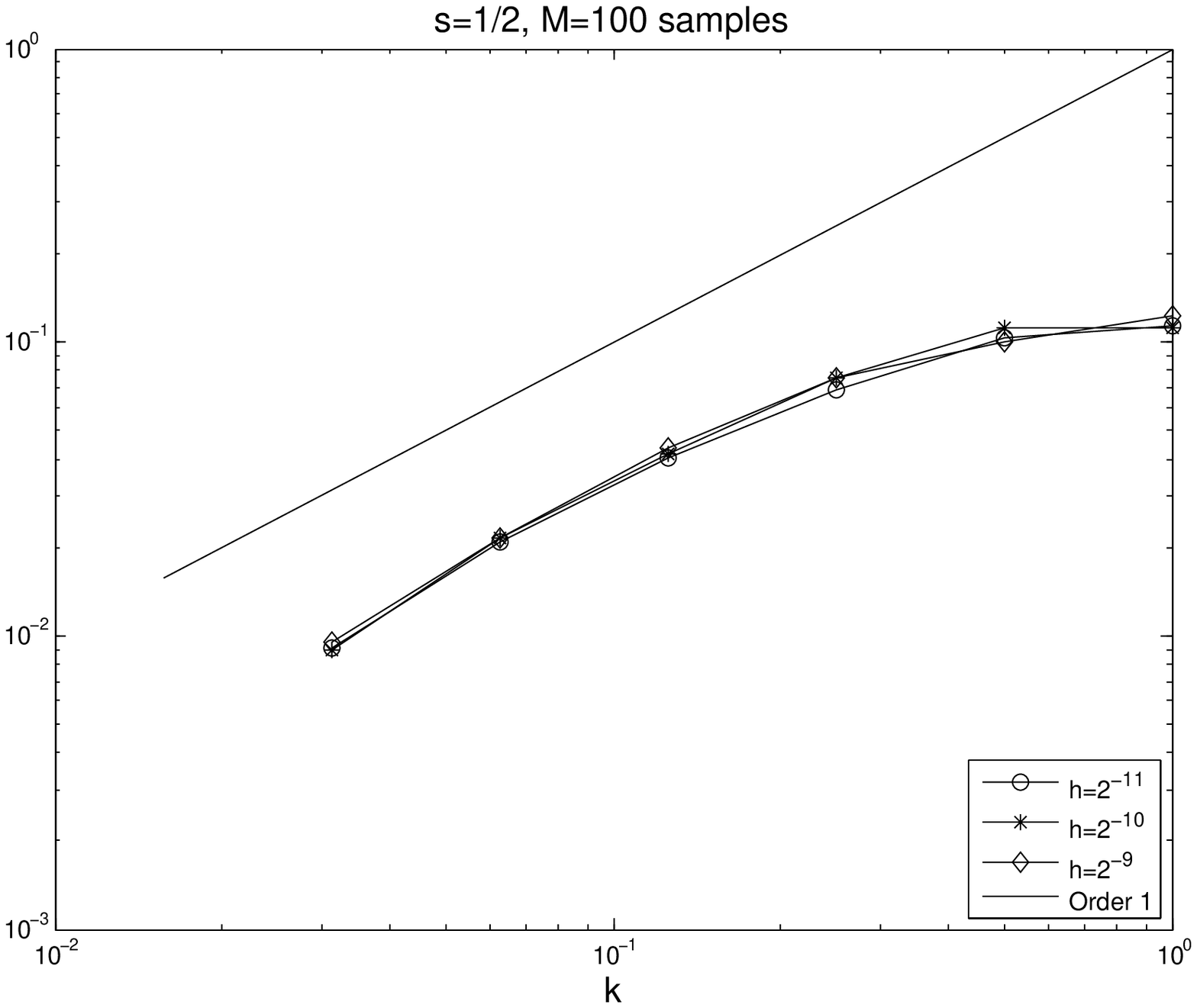}
\caption{Temporal errors: The $L_2$-error in the first component 
decreases with order $k^\beta$ and is independent of the mesh-grid $h$.} 
\label{fig:errork}
\end{center}
\end{figure}

Next, we compare our time integrator with the following 
classical numerical schemes for stochastic differential equations. 
When applied to the wave equation in the form \eqref{swe2}, these schemes 
are: 
\begin{enumerate}
\item The backward Euler-Maruyama scheme 
$X^{n+1}=X^n+kAX^{n+1}+B\Delta W^n$, 
see for example \cite{kp92} or \cite{mt04}. 
The strong rate of convergence for this method is $\bigo{k^{\min(\beta/2,1)}}$, 
see \cite[Theorem 4.12]{KLLweaktime}.
\item A stochastic version of the St\"ormer-Verlet scheme, 
writing $X=[X_1,X_2]^T$,  
\begin{equation*}
\begin{split}
X^{n+1/2}_2 & = X^n_2 + \frac{k}{2}\Lambda X^n_1 + W(t_{n+1/2})-W(t_n), \\
X^{n+1}_1 & = X^n_1 + kX^{n+1/2}_2, \\
X^{n+1}_2 & = X^{n+1/2}_2 + \frac{k}{2}\Lambda X^{n+1}_1 + W(t_{n+1})-W(t_{n+1/2}).
\end{split}
\end{equation*}
For an application of this scheme to the Langevin equation, we refer  
to \cite{reich}. We were not able to find any references on 
the strong rate of convergence of this numerical method. 
\item The Crank-Nicolson-Maruyama scheme \cite{h03}
$$
X^{n+1}=X^{n}+\frac k2 A(X^{n+1}+X^{n})+B\Delta W^n.
$$
The strong rate of convergence is $\bigo{k^{ \min(2\beta/3,1)}}$, 
see \cite[Theorem 4.12]{KLLweaktime}.
\end{enumerate}
We apply these schemes to the finite element approximation of the linear 
problem \eqref{example} with truncated noise. Note that both 
the backward Euler-Maruyama scheme and the Crank-Nicolson-Maruyama scheme 
are implicit. 
Figure~\ref{fig:trigo_bem_sv} presents the various strong convergence rates 
of the above numerical integrators, once with white noise 
and once with correlated noise with $Q=\Lambda^{-1/2}$. 
One observes that the numerical solution 
given by the St\"ormer-Verlet method 
explodes for larger values of the step-size $k$ 
(this computation was stopped when the deterministic 
non-stable regime of the scheme was attained). For all the experiments 
we use $h_{\mathrm{exact}}=2^{-10}$ for the spatial discretisation. 
The reference solution is computed using the stochastic trigonometric 
method with the step size $k_{\mathrm{exact}}=2^{-16}$. 
Again $M=100$ samples are used. 

\begin{figure}
\begin{center}
\includegraphics*[height=7cm,keepaspectratio]
{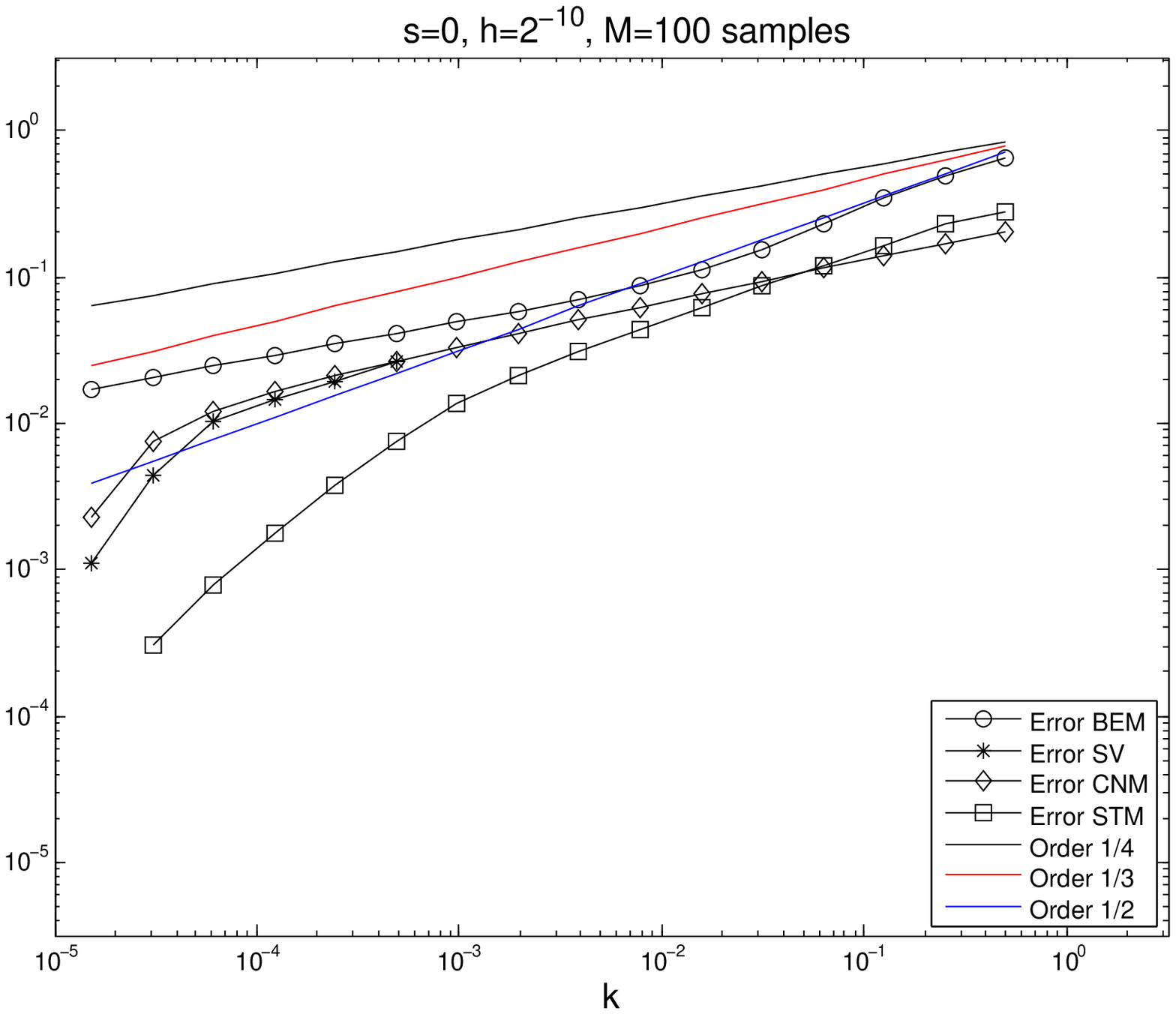}
\includegraphics*[height=7cm,keepaspectratio]
{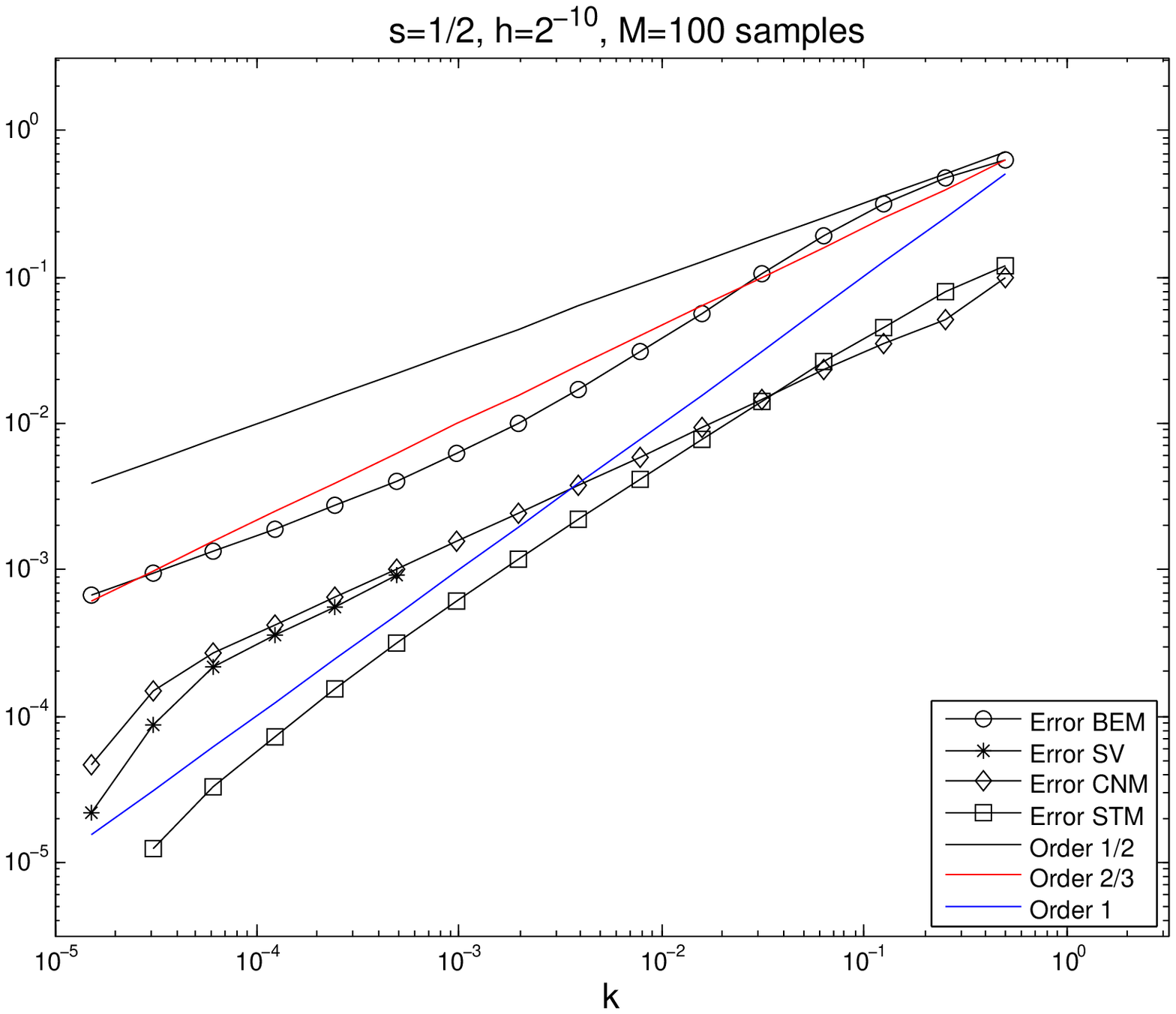}
\caption{$L_2$-error in the first component of the numerical solutions   
given by the St\"ormer-Verlet method (SV), the backward Euler-Maruyama scheme (BEM), 
the Crank-Nicolson-Maruyama scheme (CNM) and the stochastic trigonometric method (STM).} 
\label{fig:trigo_bem_sv}
\end{center}
\end{figure}

In the following numerical experiment, we are 
concerned with the trace formula of Section~\ref{sect:trace}. 
Figure~\ref{fig:energy} illustrates the trace formula of the numerical 
solution. Here, we choose $s=1/2$ and hence $\beta<1$ and display the expected 
value of the energy along the numerical solution 
of the above stochastic linear wave equation  
with mesh grids $h=0.1$ and $k=0.1$ 
on the long time interval $[0,500]$. We took $M=15000$ samples 
to approximate the expected energy of our problem. 
A comparison with other time integrators is presented 
in Figure~\ref{fig:energy2}. One notes that all these 
numerical schemes do not 
reproduce the linear growth of the expected 
energy correctly. This fact is already known for 
the backward Euler-Maruyama scheme applied to a finite-dimensional 
linear stochastic oscillator \cite{hig04}.

\begin{figure}
\begin{center}
\includegraphics*[height=7cm,keepaspectratio]
{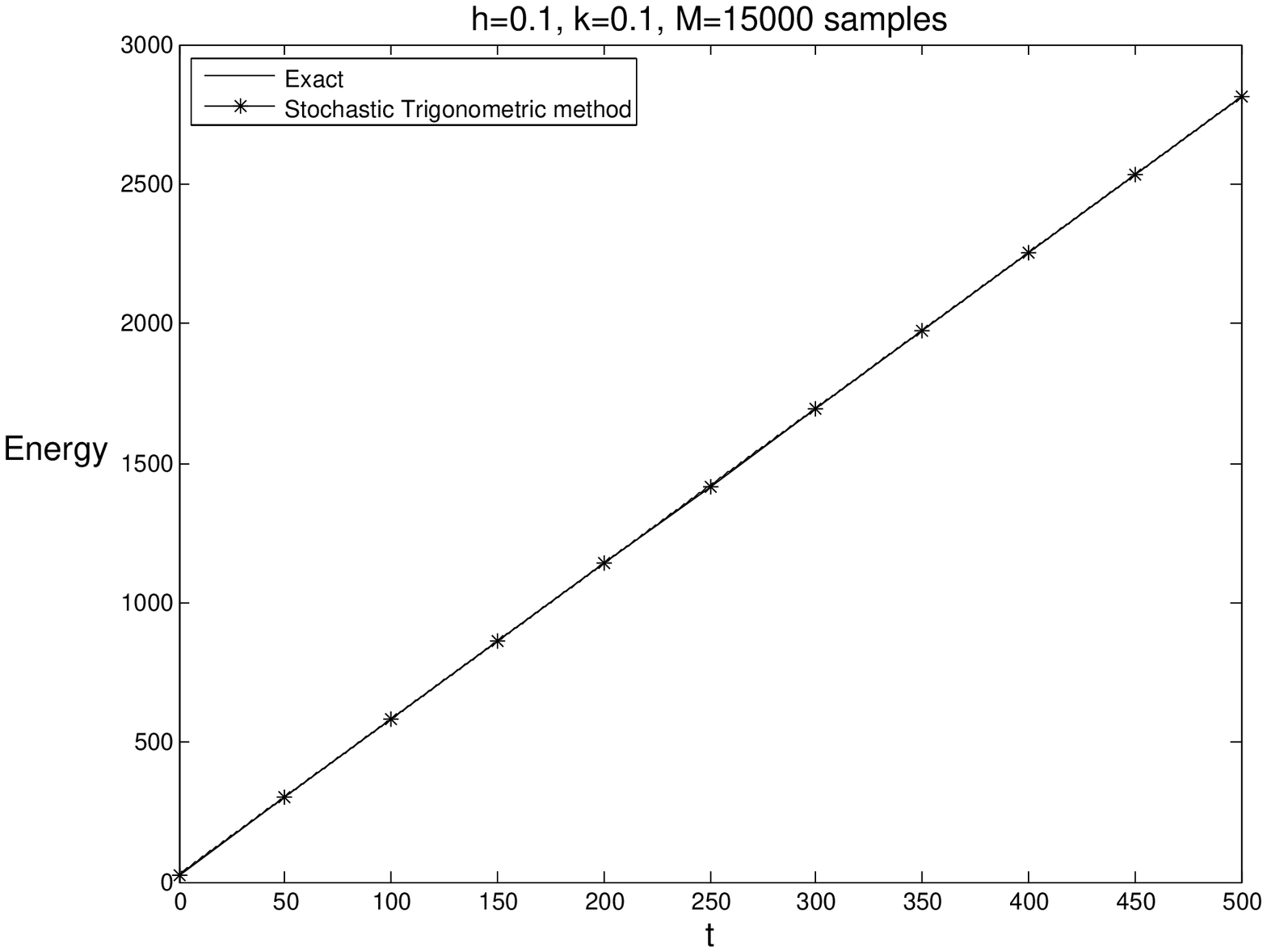}
\caption{Trace-formula: The stochastic trigonometric method preserves 
exactly the linear growth of the expected value of the energy.}
\label{fig:energy}
\end{center}
\end{figure}

\begin{figure}
\begin{center}
\includegraphics*[height=7cm,keepaspectratio]
{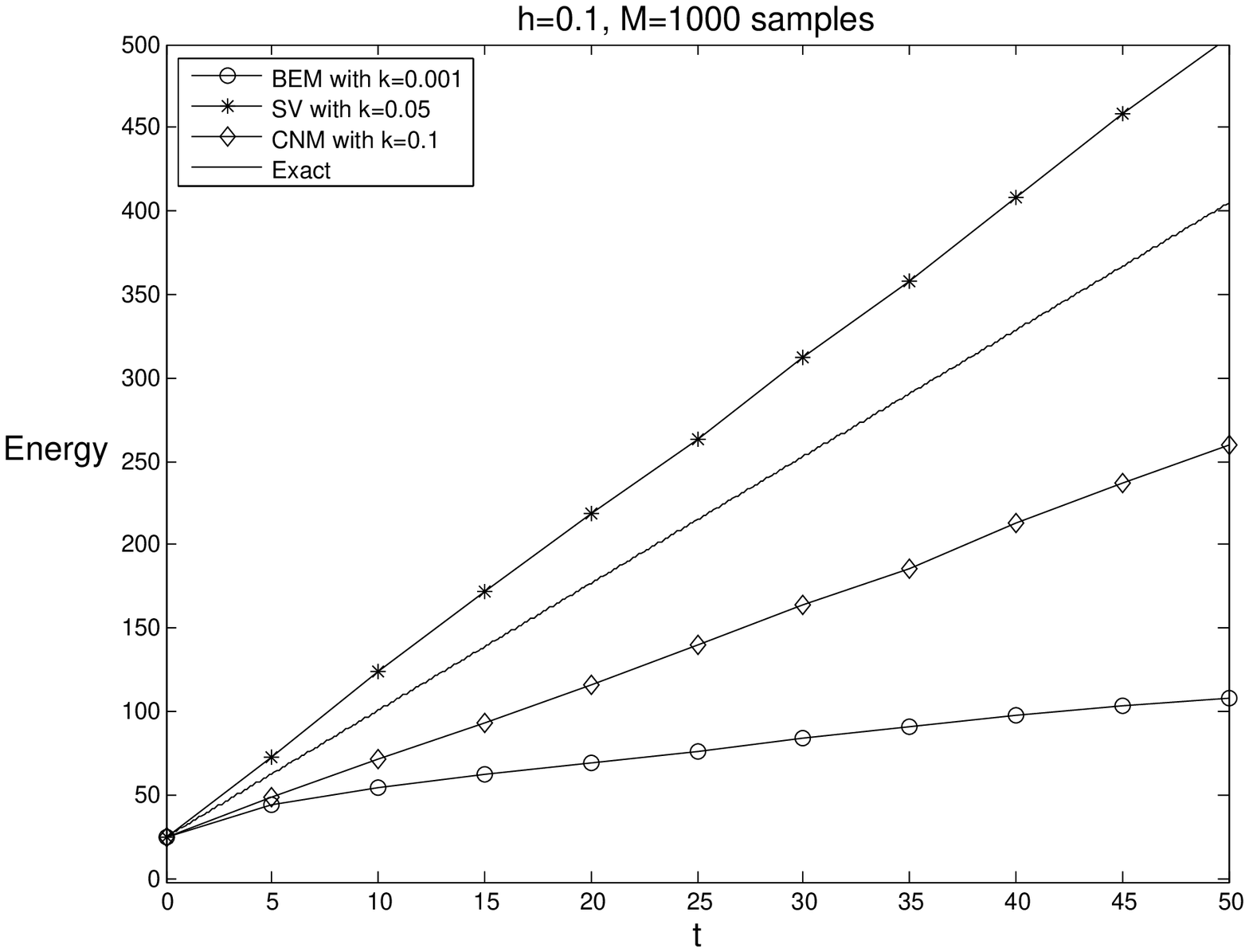}
\caption{Although using a small time step size, 
the backward Euler-Maruyama scheme (BEM) does not reproduce 
the linear growth of the expected energy. 
The St\"ormer-Verlet method (SV) and the Crank-Nicolson-Maruyama scheme (CNM) 
yield better results even with a larger time step size.}
\label{fig:energy2}
\end{center}
\end{figure}

Finally we consider a nonlinear stochastic wave equation, 
the Sine-Gordon equation driven by additive noise:
\begin{equation*}
\begin{aligned}
&\mathrm{d}\dot{u} - \Delta u\, \mathrm{d}t = -\sin(u)\, \mathrm{d}t+\mathrm{d}W, &&\quad (x,t) \in \ (0,1)\times(0,1), \\
&u(0,t)=u(1,t) = 0, &&\quad t\in(0,1), \\
&u(x,0) = 0, \ \dot u(x,0)=1_{[\frac{1}{4},\frac{3}{4}]}(x), &&\quad x \in(0,1),
\end{aligned}
\end{equation*}
where $1_I (x)$ denotes the indicator function for the interval $I$. 
The corresponding deterministic problem is studied for example in \cite{grimm}. 
We solve this problem again with a finite element method in space and in time 
we use the stochastic trigonometric method \eqref{trigo-nl} with $G(X(t))=-\sin(X(t))$ 
and the filter functions proposed in \cite{grho}:
\begin{equation*}
\psi(\xi) = \sinc^3(\xi), \ \ \phi(\xi) = \sinc(\xi), \ \ \psi_0(\xi) = \cos(\xi)\sinc^2(\xi), 
\ \ \psi_1(\xi) = \sinc^2(\xi), 
\end{equation*}
where $\sinc(\xi)=\sin(\xi)/\xi$. 
In the upper plot of Figure~\ref{fig:sine_gordon}, we show 
the expected energy of the numerical solution 
of the Sine-Gordon equation where the covariance operator 
is given by $Q=I$. 
Even for a large step-size $k=0.1$, one can observe 
the good behaviour of the numerical scheme. In the lower figure,   
we display the convergence rate for the first component 
with a covariance operator $Q=\Lambda^{-1}$.  
Again, we approximate the exact solution with a finite 
element solution and the stochastic trigonometric scheme 
using $k_{\mathrm{exact}}=2^{-6}$ and $h_{\mathrm{exact}}=2^{-9}$. 

\begin{figure}
\begin{center}
\includegraphics*[height=7cm,keepaspectratio]
{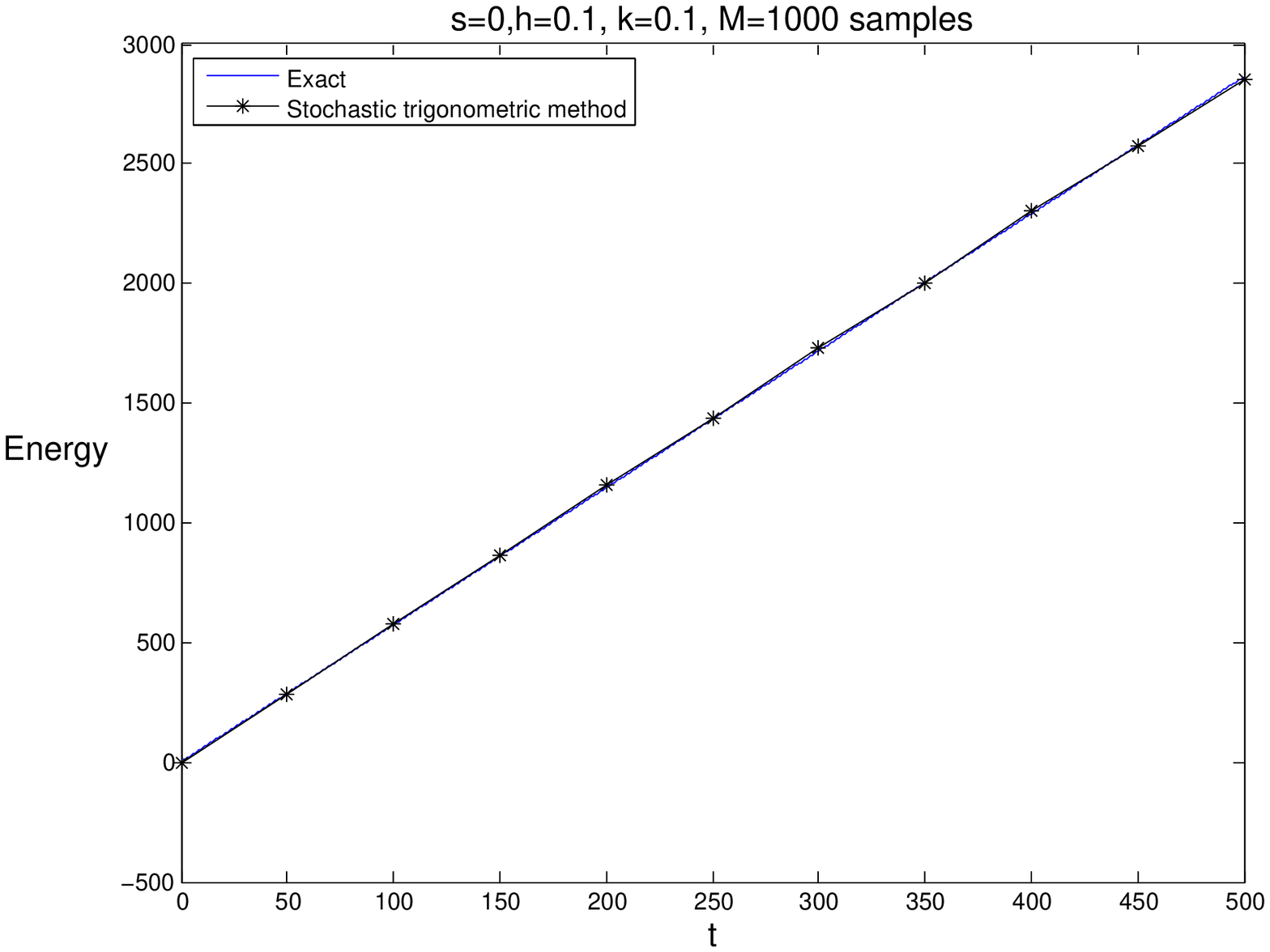}
\includegraphics*[height=7cm,keepaspectratio]
{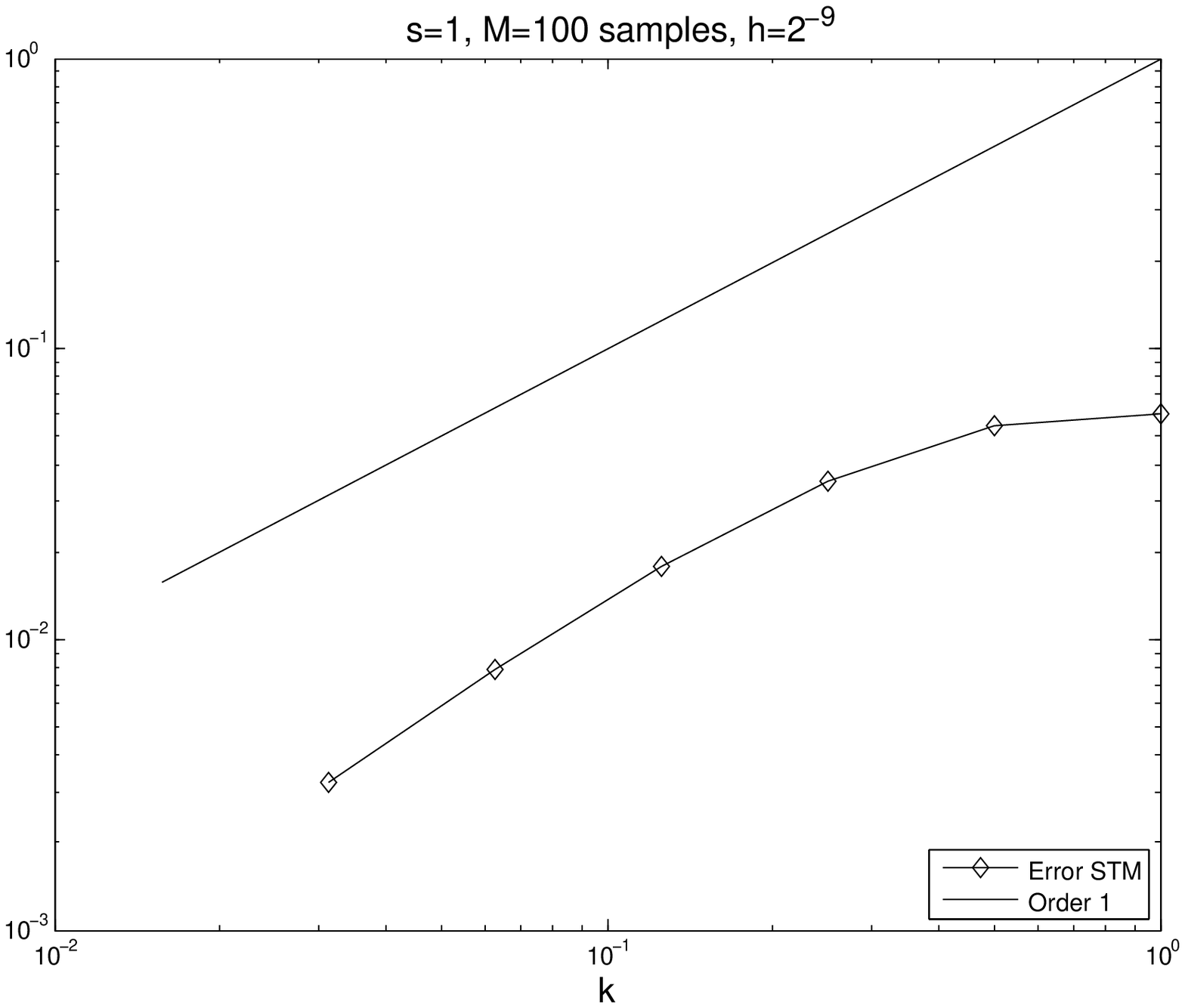}
\caption{In the nonlinear case, the stochastic trigonometric method preserves 
almost exactly the linear growth of the expected value of the energy (above figure). 
The $L_2$-error in the first component of the numerical solution given by the stochastic 
trigonometric method decreases with order $1$.}
\label{fig:sine_gordon}
\end{center}
\end{figure}

\bibliographystyle{siam}
\bibliography{biblio}


\end{document}